\documentclass{article}
\usepackage[margin=1in]{geometry}
\usepackage{amsmath, amsfonts, amssymb, amsthm}
\usepackage{graphicx}
\usepackage{color}
\usepackage{dsfont}
\usepackage{comment}
\usepackage{esint}
\usepackage[utf8]{inputenc}
\usepackage{amsmath}
\usepackage{euler}
\usepackage{textcomp}
\usepackage{gensymb}
\usepackage{subcaption}
\usepackage{caption}
\usepackage{enumitem}
\usepackage{bbm}
\usepackage{tikz}
\usepackage{tikz-cd}
\usepackage{subcaption}
\usepackage{kbordermatrix}
\usepackage{fancyvrb}
\usepackage[doi=false,isbn=false,url=false,eprint=false]{biblatex}
\usepackage{hyperref}
\usepackage[capitalise]{cleveref}
\usepackage{tikz}
\usetikzlibrary{matrix}

\newtheorem{theorem}{Theorem}[section]
\newtheorem{corollary}[theorem]{Corollary}
\newtheorem{lemma}[theorem]{Lemma}
\newtheorem{proposition}[theorem]{Proposition}
\newtheorem{conjecture}[theorem]{Conjecture}
\theoremstyle{definition}
\newtheorem{definition}[theorem]{Definition}
\newtheorem{rem}[theorem]{Remark}
\newtheorem{example}[theorem]{Example}
\newtheorem{claim}[theorem]{Claim}

\newcommand{\nc}{\newcommand}
\nc{\on}{\operatorname}

\newcommand{\Z}{\mathbb{Z}}
\newcommand{\R}{\mathbb{R}}

\newcommand{\Q}{\mathbb{Q}}
\newcommand{\C}{\mathbb{C}}

\newcommand{\FF}{\mathbb{F}}
\newcommand{\F}{\mathbb{F}}

\newcommand{\Img}{\on{Im}}

\newcommand{\coker}{\on{coker}}

\newcommand{\Syl}{\on{Syl}}

\renewcommand{\kbldelim}{(}
\renewcommand{\kbrdelim}{)}

\title{Sandpile Groups of Cayley Graphs of $\FF_2^r$}
\author{Jiyang Gao\\
{\tt\small jgao@math.harvard.edu}
\and
Jared Marx-Kuo\\
{\tt\small jmarxkuo@stanford.edu}
\and
Vaughan McDonald\\
{\tt\small 
vkm@stanford.edu}
\and
Chi Ho Yuen\\
{\tt\small chyuen@math.nctu.edu.tw}
}
\date{\today}

\makeatletter
\newcommand{\subjclass}[1]{%
  \let\@oldtitle\@title%
  \gdef\@title{\@oldtitle\footnotetext{\emph{MSC.} #1}}%
}
\newcommand{\keywords}[1]{%
  \let\@@oldtitle\@title%
  \gdef\@title{\@@oldtitle\footnotetext{\emph{Keywords.} #1}}%
}
\makeatother


\begin{document}
 
\maketitle

\begin{abstract}
The \textbf{sandpile group} of a connected graph $G$, defined to be the torsion part of the cokernel of the graph Laplacian, is a subtle graph invariant with combinatorial, algebraic, and geometric descriptions.
Extending and improving previous works on the sandpile group of hypercubes, we study the sandpile groups of the Cayley graphs of $\mathbb{F}_2^r$, focusing on their poorly understood Sylow-$2$ component.
We find the number of Sylow-$2$ cyclic factors for ``generic'' Cayley graphs and deduce a bound for the non-generic ones.
Moreover, we provide a sharp upper bound for their largest Sylow-$2$ cyclic factors. In the case of hypercubes, we give exact formulae for the largest $n-1$ Sylow-$2$ cyclic factors.
Some key ingredients of our work include the natural ring structure on these sandpile groups from representation theory, and calculation of the $2$-adic valuations of binomial sums via the combinatorics of carries.
\end{abstract}
\tableofcontents

\section{Introduction and Notation}

Let $G = (V, E)$ be a connected graph on $n$ (ordered) vertices with no loops.  
The \emph{Laplacian $L(G)$} of graph $G$ is the $n \times n$ matrix with entries
\[ 
L(G)_{u,v} = \begin{cases}
\deg(u) & u = v \\
- m(u,v) & u \neq v,
\end{cases}
\]
where $m(u,v)$ is the number of edges between $u$ and $v$.
The Laplacian is an important matrix associated with the graph that has been studied from different perspectives, ranging from more theoretical aspects to real life algorithms \cite{SpectraBook,LaplacianBook}. 
Among these perspectives, as $L(G)$ is an integer matrix, we can study its properties as a linear map of $\Z$-modules $L(G): \Z^n \to \Z^n$.

From the definition of $L(G)$ and the assumption on $G$, $\ker_{\Z} L(G)$ is spanned by the vector $(1,1, \dots, 1)^T$. Therefore $\Img L(G)$ is some codimension $1$ sublattice of $\Z^n$. It follows that the cokernel can be written as
\[
\coker{L(G)} = 
\Z^n/\Img L(G) \cong \Z \oplus K(G),
\] 
where $K(G)$ is a finite abelian group, known as the \emph{sandpile group} (also {\em critical group} or {\em Jacobian} in the literature) of $G$. Kirchhoff's Matrix--Tree Theorem shows that the size of $K(G)$ is the number of spanning trees of $G$. $K(G)$ is a subtle isomorphism invariant of a graph \cite{SandpileBook,ChipBook}, and is our main object of study. We are interested in understanding $K(G)$ for Cayley graphs of the group $\FF_2^r$, with an emphasis on their Sylow-$2$ subgroups.

A motivation for studying this family of graphs is their connection with representation theory.
Roughly speaking (see \cref{sec:background} for the details in our setting and see \cite{duceyjalil} for the more general case), for a Cayley graph $G$ of a finite abelian group and a prime $p$ that does not divide the order of the group, one can use the complex character table of the group to diagonalize $L(G)$ so that its eigenvalues are expressed as character sums, then an algebraic argument shows that one can read off the $p$-adic valuations of the invariant factors of $K(G)$ from these eigenvalues.
Further generalizations have been studied for the McKay--Cartan matrices of representations of finite groups \cite{reinerquiver, Gaetzrep}, and representations of finite-dimensional Hopf algebras \cite{GHR_Hopf}.
However, in our setting, the strategy fails when $p=2$, which echoes the difficulties in modular representation theory\footnote{The failure of the strategy comes from the fact that $|G|$ is not invertible in $\Z_p$ (or any finite extension thereof) when $p\mid |G|$, analogous to the situation when $|G|$ is not invertible in characteristic $p$.}.

Nevertheless, partial progress has been made along this direction, particularly for the hypercube graphs $Q_n$, which are the Cayley graphs of $\FF_2^n$ with respect to the standard basis $\{e_1,\dots, e_n\}\subseteq \FF_2^n$. 
In 2003, Bai \cite{bai} determined the Sylow-$p$ subgroups of $K(Q_n)$ for $p \neq 2$ with an elementary argument, and derived formulae for the number of Sylow-$2$ cyclic factors and the number of $\Z/2\Z$-factors in $K(Q_n)$.
In 2015, Chandler et al. \cite{chandler2017smith} determined the cokernel of the \emph{adjacency matrix} (known as the \emph{Smith group}) of $Q_n$ in terms of $n$. 
Anzis and Prasad \cite{anzisprasad} made the next step in 2016 by bounding the largest Sylow-$2$ cyclic factor of $K(Q_n)$.

We also note the other work by Chandler et al. \cite{chandler2017paley}, where they determined the critical groups of Paley graphs, a special family of Cayley graphs of $\FF_q$ for prime powers $q=p^k\equiv 1 \pmod 4$.
They were able to approach the Sylow-$p$ subgroup of the critical group via Jacobi sums, and by applying Stickelberger's theorem, they further reduced the problem to the combinatorics of ``carries'' when adding two $p$-ary numbers. Our approach of \cref{hypercubecyclicc} thus has a similar flavor of their work.

We now summarize our contributions in this paper.
It is worth emphasizing that many of our arguments are based on the natural ring structure of $K(G)$ coming from representation theory \cite{reinerquiver} and the corresponding polynomial algebra.
We begin by defining a \emph{generic} set of generators $M$ for a Cayley graph of $\F_2^r$ as set of generators whose sum is nonzero (see \cref{definition: generic}). For example, hypercubes have generic generating sets.  We then have the following result that generalizes Bai's result \cite[Theorem~1.1]{bai}.
\begin{theorem}
\label{dMgeneric}
Suppose that $M$ is generic. Then the number of Sylow-2 cyclic factors of $K(G(\FF_2^{r}, M))$ is $2^{r-1} - 1$.
\end{theorem}

On the other hand, we provide a qualitative result for non-generic cases that complements \cref{dMgeneric}:

\begin{theorem} \label{dMnon-generic}
Suppose that $M$ is not generic. Then the number of Sylow-2 cyclic factors of $K(G(\FF_2^{r}, M))$ is strictly larger than $2^{r-1} - 1$.
\end{theorem}

We then investigate the size of the largest Sylow-$2$ cyclic factor in $K(G(\FF_2^{r}, M))$. 
Suppose that $K(G)$ has the invariant factor decomposition
$K(G) = \bigoplus_{i = 1}^{d}\Z/c_i(G)\Z$, where $c_d(G) \mid c_{d-1}(G)\mid c_{d-2}(G)\mid \dots \mid c_1(G)$, thus $c_1(G)$ is the largest cyclic factor. 
Define $v_2(x)$, the $2$-adic valuation of $x$, as the largest $n$ such that $2^n\mid x$. We adopt the methods of Anzis and Prasad to extend and improve their upper bound to arbitrary Cayley graphs $G$ of $\FF_2^r$.
\begin{theorem}
The largest cyclic factor $\Z / c_1(G) \Z$ of $K(G(\FF_2^r,M))$ satisfies
\[
v_2(c_1(G)) \leq \lfloor \log_2 n\rfloor + r - 1,
\]
where $n$ is the number of generators in $M$.
\end{theorem}
In the case of $Q_n$, we go further to explicitly determine the top cyclic factor.
\begin{theorem} \label{hypercubecyclicc}
For $n \ge 2$, the largest cyclic factor $\Z / c_1(Q_n)\Z$ of $K(Q_n)$ satisfies
\[
v_2(c_1(Q_n))=\max\big\{\max_{x<n}\{v_2(x)+x\}, v_2(n)+n-1\big\}.
\]
\end{theorem} 
We then continue to determine the $2^\text{nd}$ through $(n-1)^\text{th}$ cyclic factors, and conjecture a formula for the $n^\text{th}$ factor (see \cref{thm:2ndtopcyclicQn} and \cref{conj:nthtopcyclicQn}).
Finally, we conclude with some remaining conjectures about the structure of $K(G)$ in Section~5.

\begin{rem} 
\label{importantremark}
\cref{dMnon-generic} was posed as a conjecture in an earlier version of this paper.
Iga et al. proved partial results of the conjecture \cite{IKKY} (see the discussion at the end of \cref{section:2_rank} for more background) before the conjecture was solved in full generality in the current version.
\end{rem}

\subsection{Background and Previous Results} \label{sec:background}

We first define what a Cayley graph is in our context.
Given a multiset of nonzero generators
\[M = \begin{pmatrix}
| &  & | \\
v_1 & \cdots & v_n \\
| &  & |
\end{pmatrix}
\]
of $\FF_2^r$, we form the \emph{Cayley graph} $G = G(\FF_2^r, M)$ with vertex set $V = \FF_2^r$ and multiedge set $\big\{(w, w+v_i)\big\}$ for $w \in V$ and $v_i \in M$.  $G$ is connected because $M$ is a generating set, and there are no loops since all $v_i \neq 0$. As addition is performed in $\FF_2^r$, we also have $w + v_i + v_i = w + 2v_i = w$. Therefore, we can think of this graph as undirected. If we index the matrix representation of $L(G)$ by the binary tuples $u, v \in \FF_2^r$ then we can say that
\[
L(G)_{u,v} = \begin{cases}
n & u = v \\
-(\# \text{ of generators, $v_i$, such that } u + v_i = v) & u \neq v
\end{cases}
\]
since $G$ is an $n$-regular graph. 
\begin{example}
\label{ex:cayleyex}
We give three examples of Cayley graphs in the case $r = 3$.
\begin{enumerate}
    \item The cube $Q_3$ with $M= \begin{pmatrix}
    1 & 0 & 0\\
    0 & 1 & 0\\
    0 & 0 & 1
    \end{pmatrix}$ is shown in \cref{fig:1a}.
    \item Define $Q_{3,\on{vert}}$ to be the Cayley graph with associated generating set 
    \[
    M = \begin{pmatrix}
    1&0&0&0\\
    0&1&0&0\\
    0&0&1&1
    \end{pmatrix}.\] Then the graph is a cube with two sets of vertical edges shown in \cref{fig:1b}.
    \item We can recover the complete graph $K_{8}$ by setting 
    \[
    M = \begin{pmatrix}
    1&0&0&1&0&1&1\\
    0&1&0&1&1&0&1\\
    0&0&1&0&1&1&1
    \end{pmatrix}.
    \]
\end{enumerate}
\end{example}
\begin{figure}[ht!]
\centering
\begin{subfigure}[b]{.4\linewidth}
\centering
\begin{tikzpicture}[scale=2,v/.style={fill, circle, outer sep=0pt, inner sep=0pt, minimum size=4pt}, e/.style={thick}]
  \node[v] (001) at (0,0,1) {};
  \node[v] (011) at (0,1,1) {};
  \node[v] (101) at (1,0,1) {};
  \node[v] (111) at (1,1,1) {};
  \node[v] (000) at (0,0,0) {};
  \node[v] (010) at (0,1,0) {};
  \node[v] (100) at (1,0,0) {};
  \node[v] (110) at (1,1,0) {};
  \draw[e](110)--(010)--(011)--(111)--(110)--(100)--(101)--(001)--(011);
  \draw[e](111)--(101);
  \draw[e](100)--(000)--(010);
  \draw[e](000)--(001);
\end{tikzpicture}
\caption{Hypercube $Q_3$}\label{fig:1a}
\end{subfigure}
\begin{subfigure}[b]{.4\linewidth}
\centering
\begin{tikzpicture}[scale=2,v/.style={fill, circle, outer sep=0pt, inner sep=0pt, minimum size=4pt}, e/.style={thick}]
  \node[v] (001) at (0,0,1) {};
  \node[v] (011) at (0,1,1) {};
  \node[v] (101) at (1,0,1) {};
  \node[v] (111) at (1,1,1) {};
  \node[v] (000) at (0,0,0) {};
  \node[v] (010) at (0,1,0) {};
  \node[v] (100) at (1,0,0) {};
  \node[v] (110) at (1,1,0) {};
  \draw[e](101)--(001)--(000)--(100)--(101);
  \draw[e](111)--(011)--(010)--(110)--(111);
  \draw[e,bend left=15] (111) to (101) to (111);
  \draw[e,bend left=15] (110) to (100) to (110);
  \draw[e,bend left=15] (011) to (001) to (011);
  \draw[e,bend left=15] (010) to (000) to (010);
\end{tikzpicture}
\caption{Cayley Graph with multiedges}\label{fig:1b}
\end{subfigure}
\caption{Examples of Cayley Graphs in dimension $r=3$}\label{fig:1}
\end{figure}
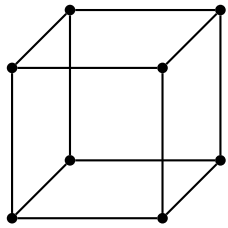
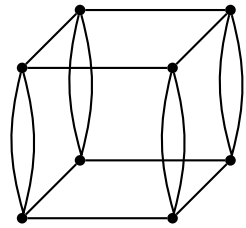

We are interested in $K(G)$ for such graphs.
Kirchhoff's Matrix--Tree Theorem tells us that if we label the eigenvalues of $L(G)$ as $0 = \lambda_1 < \lambda_2 \le \lambda_3 \le \dots \le \lambda_{2^r}$, then $|K(G)| = \frac{1}{2^r}\prod_{i = 2}^{2^r}|\lambda_i|$, which is also the number of spanning trees of $G$ \cite{StanleyBook}. On the other hand, by the structure theorem for finite abelian groups, $K(G) \cong \bigoplus_p \oplus_{e \geq 1} (\Z / p^e \Z)^{m(p^e)}$, where $m(p^e)$ is the multiplicity of $\Z/p^e \Z$ in $K(G)$.
Thus, we can try to determine the group prime by prime in terms of its Sylow subgroups.
We note that, while eigenvalues are useful in the study of $K(G)$ for Cayley graphs $G$, the spectrum of $L(G)$ provides rather little information about $K(G)$ in general; some of these connections can be found in \cite{rushanan}.

We detail some basic properties about the sandpile group of an arbitrary Cayley graph of $\FF_2^r$.
As mentioned in the introduction, much of this generalizes to finite abelian groups with slightly more complicated algebraic formalism, cf. \cite{duceyjalil}.
Fixing $r$, when considering matrices over $\R$, there is an eigenbasis for all of $L(G(\FF_2^r, M))$ simultaneously. 
\begin{definition}
For $u \in \FF_2^r$, define
\[
f_u = \sum_{v \in \FF_2^r} (-1)^{u \cdot v} e_v
\]
where $e_v$ is the standard basis vector $(0, \dots, 0, 1, 0, \dots, 0)^T \in \R^{2^r}$ with the only $1$ in the $v$-coordinate.
\end{definition}

\begin{lemma}\label{changeofbasis}
$e_u$'s can be expressed in terms of $f_v$'s as follows:
\[e_u = \frac{1}{2^{r}} \sum_{v \in \FF_2^r} (-1)^{u \cdot v} f_v.\]
Hence, $\{f_u\}$ is a basis for for $\R^{2^r}$.
In fact, it is an orthogonal eigenbasis for any generating set $M$, with the eigenvalue corresponding to $f_u$ being $\lambda_{u, M} = n - \sum_{i = 1}^n (-1)^{u \cdot v_i}$.
\end{lemma}

\begin{proof}
For a proof of the second statement and a more thorough exposition of the setting of this problem, see \cite[Chapter 2]{StanleyBook}. We only prove the first statement here.

We have $\sum_{v} (-1)^{u \cdot v} f_v=\sum_{v} (-1)^{u \cdot v} \big[\sum_{w} (-1)^{v\cdot w} e_w\big]=\sum_{w} \big[\sum_{v} (-1)^{(u+w)\cdot v}\big] e_w$. For $w=u$, all $(-1)^{(u+w)\cdot v}$'s are 1, so the coefficient of $e_w$ is $2^r$; for $w\neq u$, say they differ over the $i$-th coordinate, then we can pair up the elements of $\FF_2^r$ by fliping their $i$-th coordinate, the sum of $(-1)^{(u+w)\cdot v}$ associated with a pair is 0, so the coefficient of $e_w$ is 0.
\end{proof}

One can use this information to determine the Sylow-$p$ structure for odd primes $p$.
\cref{changeofbasis} implies we can diagonalize $L(G)$ to a matrix $D = \text{diag}(\lambda_{u,M}, u \in \FF_2^r)$ using matrices whose entries are from $\Z\left[\frac{1}{2}\right]$, which provides a matrix equivalence that relates $D$ to the Smith Normal Form of $L(G)$ over $\Z\left[\frac{1}{2}\right]$.
This implies the following proposition.
\begin{proposition}
\label{prop:pnot2}
For any odd prime $p \neq 2$,
\[
\Syl_p K(G) = \Syl_p\left(\bigoplus_{u \in \FF_2^r\setminus \{0\}} \Z / \lambda_{u, M} \Z\right).
\]

\end{proposition}

Thus, we have a nice description of the Sylow-$p$ subgroups for $p \neq 2$ in terms of the eigenvalues, but the Sylow-$2$ subgroups exhibit a more irregular structure.
\begin{example}
These are the sandpile groups for our 3 examples in the case $r = 3$:
\begin{enumerate}
    \item The eigenvalues are $(0,2,2,2,4,4,4,6)$ and $K(Q_3) = \Z/2\Z \oplus (\Z/8\Z)^2 \oplus \Z/3\Z$
    \item The eigenvalues are $(0,2,2,4,4,6,6,8)$ and $K(Q_{3,\on{vert}})  = (\Z/4\Z)^2 \oplus \Z/16\Z \oplus (\Z/3\Z)^2$,
    \item The eigenvalues are $(0,8,8,8,8,8,8)$ and $K(K_{8}) = (\Z/8\Z)^6$.
\end{enumerate}
\end{example}

In order to deal with the case $p = 2$, we adopt the approach of Benkart, Klivans, and Reiner \cite{reinerquiver} that gives a natural ring structure on $K(G)$:

\begin{proposition}
\label{cokerIso}
For $G=G(\FF_2^r,M)$, we have an isomorphism of $\Z$-modules
\begin{align*}\operatorname{coker }  L(G) \cong K(G) \oplus \Z \cong  \Z[x_1, x_2, \dots, x_r] / \bigg(x_1^2 - 1, \dots, x_r^2 - 1, n - \sum_{i = 1}^n \prod_{j = 1}^r x_j^{(v_i)_j}\bigg).
\end{align*}
The isomorphism is given by $e_v\leftrightarrow \prod_{j=1}^r x_j^{(v)_j}$ for all $v\in\FF_2^r$. We denote the ideal $\bigg(x_1^2 - 1, \dots, x_r^2 - 1, n - \sum_{i = 1}^n \prod_{j = 1}^r x_j^{(v_i)_j}\bigg)$ by $I(G)$.
\end{proposition}
\begin{proof}
First, there is an isomorphism of $\Z$-modules
\[\Z^{2^r}\cong \Z[x_1,\dots,x_r]/(x_1^2 - 1, \dots, x_r^2 - 1)\]
under the map $e_v\leftrightarrow \prod_{j=1}^r x_j^{(v)_j}$ for all $v\in\FF_2^r$. To prove the proposition, it suffices to examine the image of $\operatorname{im }L(G)$ under this isomorphism, and show that it is the ideal generated by $n - \sum_{i = 1}^n \prod_{j = 1}^r x_j^{(v_i)_j}$. Based on the definition of $L(G)$, the Laplacian maps $e_u$ to $n\cdot e_u-\sum_{i=1}^n e_{u+v_i}$. Under our isomorphism, this implies that $L(G)$ maps monomial $\prod_{j=1}^r x_j^{(u)_j}$ to
\[n\prod_{j=1}^r x_j^{(u)_j}-\sum_{i = 1}^n \prod_{j = 1}^rx_j^{(u)_j+(v_i)_j}=\prod_{j=1}^r x_j^{(u)_j}\left(n - \sum_{i = 1}^n \prod_{j = 1}^r x_j^{(v_i)_j}\right).\]
As a result, the image of $\operatorname{im }L(G)$ in $\Z[x_1,\dots,x_r]/(x_i^2-1)$ is the ideal generated by $n - \sum_{i = 1}^n \prod_{j = 1}^r x_j^{(v_i)_j}$.
\end{proof}

\begin{rem}
As \cite{reinerquiver} does not explicitly use the language of Cayley graphs, we explain how to deduce \cref{cokerIso} from its results, and use this opportunity to explain the representation theoretic background of it.

In the paper, the authors considered the notion of a critical group attached to a faithful representation $\gamma$ of a finite group $\Gamma$, which is the torsion part of the cokernel of the {\em extended McKay--Cartan matrix} associated with $\gamma$: the rows and columns of the matrix are indexed by the irreducible (complex) representations of $\Gamma$, with $\dim\gamma$ on the diagonal entries, and each off-diagonal $(\alpha,\beta)$-entry equals minus the multiplicity of $\beta$ in $\gamma\otimes\alpha$. This group is isomorphic to (the additive group of) the quotient of the {\em virtual representation ring} of $\Gamma$ \cite[Proposition~5.22]{reinerquiver}, where its elements correspond to the complex representations of $\Gamma$, with addition and multiplication coming from taking direct sum and tensor product of representations, respectively.

In our case, the variable $x_i$ corresponds to the irreducible representation $\FF_2^r\rightarrow\mathbb{C}$ given by $(b_1,\ldots,b_r)\mapsto (-1)^{b_i}$, so all irreducible representations of $\FF_2^r$ are in bijection with the squarefree monomials in $x_i$'s. The map $e_v\leftrightarrow \prod_{j=1}^r x_j^{(v)_j}$ then gives a bijection between elements of $\FF_2^r$ and the irreducible representations of the very group.
Under this identification, we can consider the representation $\bigoplus_{i=1}^n e_{v_i}\leftrightarrow \sum_{i = 1}^n \prod_{j = 1}^r x_j^{(v_i)_j}$ of $\FF_2^r$, which is faithful as $v_i$'s generate $\FF_2^r$ by assumption. It is routine to check that the Laplacian of $G(\FF_2^r, M)$ is exactly the extended McKay--Cartan matrix associated with this representation, and \cref{cokerIso} follows from \cite[Proposition~5.22]{reinerquiver}.
\end{rem}

\begin{example}
In the case $G = Q_n$, this polynomial ring is 
\[
\Z[x_1,,x_2,\dots, x_n]/(x_1^2-1, x_2^2-1, \dots, x_n^2-1, x_1 + x_2 + \dots + x_n-n).
\]
The ring and the ideal $I(Q_n)=(x_1^2-1, x_2^2-1, \dots, x_n^2-1, x_1 + x_2 + \dots + x_n-n)$ are symmetric under the action of $S_n$, which is an important fact that we use later.
\end{example}

The two descriptions of the sandpile group $K(G)$ in \cref{cokerIso} each has its own advantages: the {\em matrix-theoretic} definition allows us to apply linear algebra techniques, and the {\em representation-theoretic (ring-theoretic)} definition has a special multiplicative structure. We will switch between the two descriptions throughout this paper. 

\begin{rem}
It is important to note the sandpile group $K(G(\FF_2^r, M))$ is invariant under left multiplication by elements of $\text{GL}_r(\FF_2)$.
That is, given $ T \in \text{GL}_r(\FF_2)$, we have $K(G(\FF_2^r, M)) \cong K(G(\FF_2^r, T\cdot M))$.
This is because $M$ and $T\cdot M$ form Cayley graphs that are isomorphic via the bijective map ${\bf u}\mapsto T\cdot {\bf u},\forall {\bf u}\in\FF_2^r$ of vertices.
As a result, we are free to replace a given $M$ by any representative from the $\text{GL}_r(\FF_2)$-orbit of $M$.
\end{rem}

{\bf Convention:} It is often convenient to label the elements of $\FF_2^r$ as $w_0,\ldots,w_{2^r-1}$, where the index of the vector $(d_1,\ldots,d_r)^T$ is the binary number $\overline{d_1\ldots d_r}$.
For the generator set $M$, we denote the multiplicity of $w_i$ by $\mu_i$, so $n=\sum_{i=0}^{2^r-1} \mu_i$. In \cref{topcyclicQnsection}, we also label the elements of $\FF_2^r$ as $w_S$ for all subsets $S\subset [r]$, where $w_S$ has entry $1$ at indices $s\in S$ and $0$ elsewhere. For example, when $r=5$, $w_{\{1,3,4\}}=(1,0,1,1,0)^T\in \FF_2^r$.
\section{The Number of 2-Cyclic Factors} \label{section:2_rank}

In this section we wish to compute the number of $2$-cyclic factors appearing in the sandpile group for $G = G(\FF_2^r, M)$. Given a sandpile group
\[
K(G) \cong \bigoplus_{p} \bigoplus_{e \geq 1} (\Z / p^e \Z)^{m(p^e)},
\]
we can tensor with $\Z/2\Z$ to get
\[
K(G) \otimes (\Z / 2\Z) \cong \bigoplus_{e \geq 1} (\Z / 2^e \Z)^{m(2^e)} \otimes \Z / 2\Z \cong (\Z / 2\Z)^{\sum_{e}m(2^e)},
\]
where we used that $(\Z/p^e \Z) \otimes (\Z/2\Z) = 0$ for $p \neq 2$ and $(\Z / 2^e \Z) \otimes (\Z/2\Z) = \Z/2\Z$ for all $e \geq 1$. We define
\[
d(M):= \sum_{e \geq 1} m(2^e),
\]
which is the number of even invariant factors.

As stated in the introduction, we establish a dichotomy of Cayley graphs in terms of their values of $d(M)$.
First, we define the genericity condition on the generating set $M$ that only depends on the parity of $\mu_i$'s:
\begin{definition}
\label{definition: generic}
We say that $M$ is \textbf{generic} if 
\[
\sum_{i = 1}^{n} v_n \neq \vec{0}.
\]
\end{definition}

\begin{example}
In the examples for $r = 3$ in \cref{ex:cayleyex}, $Q_3$ and $Q_{3,\on{vert}}$ are generic, but $K_8$ is \textbf{not} generic.
\end{example}

We prove a few useful observations.

\begin{lemma} \label{lem:even_equal_2rank}
$d(M)$ equals $2^r-1$ minus the rank of $L(G(\FF_2^r,M))$ over $\FF_2$.
\end{lemma}

\begin{proof}
Tensoring the exact sequence $\Z^{2^r}\xrightarrow{L}\Z^{2^r}\rightarrow K(G)\oplus\Z\rightarrow 0$ with $\Z/2\Z$ yields $(\Z/2\Z)^{2^r}\xrightarrow{L}(\Z/2\Z)^{2^r}\rightarrow (\Z/2\Z)^{d(M)+1}\rightarrow 0$, hence $(\Z/2\Z)^{2^r}\cong \Img L\oplus(\Z/2\Z)^{d(M)+1}$.
Taking the dimension of both sides proves the lemma.
\end{proof}

\begin{proposition}
\label{prop: multiplicitiesindependentmod2}
$d(M)$ only depends on the parity of the multiplicities $\mu_i$'s.
\end{proposition}
\begin{proof}
The rank of $L(G(\FF_2^r,M))$ over $\FF_2$ only depends on the parity of its entries, which only depends on the parity of $\mu_i$'s.
The statement follows from \cref{lem:even_equal_2rank}.
\end{proof}

\begin{proposition}
Let $(\mu_1,\dots, \mu_{2^{r} - 1})$ be the multiplicities of the generators associated to $M$, and assume not all $\mu_{i}$ have the same parity. Let $M'$ be the generating set with multiplicities $(\mu_1 + 1, \dots, \mu_{2^r -1} + 1)$. Then $d(M) = d(M')$.
\label{prop:evenswitchtoodd}
\end{proposition}

In the following proof, we use the term {\em support} in two different senses, which should be clear from the context.
For a vector $w_i \in \FF_2^r$, its support is the set $S_i = \{j : (w_i)_j = 1\}$; for a polynomial, its support consists of all monomials whose coefficients are nonzero.

\begin{proof}
Apply the change of variables $u_i := x_i - 1$ to the presentation in \cref{cokerIso}, and tensor the ring with $\Z/2\Z$.
We thus obtain the ring
\[
R := (\Z/2\Z)[u_1,\dots, u_r]/\left(u_1^2,\dots, u_r^2, p(u_1,u_2,\dots, u_r):= n - \sum_{i = 1}^{2^r - 1} \mu_i \prod_{j = 1}^r (u_j+1)^{(w_i)_j}\right),
\]
which is isomorphic to $(K(G) \oplus \Z)\otimes(\Z/2\Z)$.

For $M'$, we similarly obtain the ring
\[
R' := (\Z/2\Z)[u_1, \dots, u_r]/(u_1^2, \dots, u_r^2, p'(u_1, \dots, u_r)),
\]
where $p'(u_1, \dots, u_r) = n+(2^r-1)-\sum_{i = 1}^{2^r - 1} (\mu_i+1)\prod_{j = 1}^r (u_j+1)^{(w_i)_j}$.

We want to show that $p'(u_1,\ldots,u_r)\in I(G)$ and $p(u_1,\ldots,u_r)\in I(G')$, as they imply $I(G)=I(G')$, hence $R=R'$ and $d(M)= \dim_{\Z/2\Z}R - 1=\dim_{\Z/2\Z}R' - 1=d(M')$.
We claim that $p(u_1,\dots u_r) - p'(u_1,\dots, u_r) = 1 + \sum_{i = 1}^{2^r - 1}\prod_{j = 1}^{r}(u_j + 1)^{(w_i)_j} \in I(G) \text{ and } I(G')$, which would give the desired containment.
Set $u^K = \prod_{k \in K}u_k $. Then expanding yields
\[
\prod_{j = 1}^r(u_j + 1)^{(w_i)_j} = \sum_{K \subseteq S_i} \prod_{k \in K} u_k = \sum_{K \subseteq S_i} u^K.
\]
From this we get
\begin{align*}
1+\sum_{i = 1}^{2^r - 1} \prod_{j = 1}^r (u_j + 1)^{(w_i)_j} &= 1+ \sum_{i = 1}^{2^r - 1} \sum_{K \subseteq S_i} u^K\\
&= 1+\sum_{K \subseteq \{1, \dots, r\}} \#( i : K \subseteq S_i) \cdot u^K\\ 
&= 
(1+(2^r-1))+
\sum_{K \subseteq \{1, \dots, r\}, K\neq\emptyset} 2^{r - |K|} \cdot u^K\\
&= u_1\ldots u_r.
\end{align*}
Thus it remains to show that $u_1\dots u_r\in I(G)$; since our only assumption on $G$ is that $\mu_i$'s do not have the same parity, which $G'$ satisfies by construction, the argument below also shows $u_1\dots u_r\in I(G')$.

From the assumption on $\mu_i$'s, $p(u_1,\ldots,u_r)$ contains some monomial that is not $u_1\ldots u_r$: if $\mu_{2^r-1}$ is odd, then we pick, among all $w_j$'s whose multiplicity is even, any $w_i$ whose $S_i$ is inclusion-maximal.
There are $2^{r-|S_i|}-1$ elements in $\FF_2^r$ whose supports properly contain $S_i$, and their multiplicities are odd.
Since $\prod_{j=1}^r (u_j + 1)^{(w_{i'})_j}$ contains the monomial $u^{S_i}$ if and only if $S_i\subset S_{i'}$, $u^{S_i}$ is a monomial in the support of $p(u_1,\ldots,u_r)$.
The case when $\mu_{2^r-1}$ is even follows from an analogous argument.

Next we pick, among all $w_j$'s such that $u^{S_j}$ is in the support of $p(u_1,\ldots,u_r)$, any $w_i$ whose $S_i$ is inclusion-minimal.
In particular, for any other $w_j$ such that $u^{S_j}$ is in the support of $p(u_1,\ldots,u_r)$, $S_j\setminus S_i$ is non-empty, and $u^{S_j}\cdot\frac{u_1 \cdots u_r}{u^{S_i}}$ is not squarefree.
Now
\[
I(G)\ni p(u_1,\ldots,u_r)\cdot\frac{u_1\ldots u_r}{u^{S_i}}
= u_1 \cdots u_r + [p(u_1, \dots, u_r) - u^{S_i}]\cdot\frac{u_1 \cdots u_r}{u^{S_i}},
\]
using the fact that every term in $[p(u_1, \dots, u_r) - u^{S_i}]\cdot\frac{u_1 \cdots u_r}{u^{S_i}}$ is the multiple of some $u_j^2$ hence in $I(G)$, we have $u_1\ldots u_r\in I(G)$ as wanted.
\end{proof}  

\begin{example}
Clearly \cref{prop:evenswitchtoodd} follows from \cref{dMgeneric} in the generic case, since adding multiplicity one to each vector does not change the sum.
Consider then the case 
\setcounter{MaxMatrixCols}{20}
\begin{equation*}
M = \begin{pmatrix}
1 & 0 & 0 & 1\\
0 & 1 & 0 & 1\\
0 & 0 & 1 & 1
\end{pmatrix}\\
\end{equation*}
\begin{equation*}
    N = \begin{pmatrix}
1&1&0&0&0&0&1&1&1&0&1\\
0&0&1&1&0&0&1&1&1&1&0\\
0&0&0&0&1&1&1&1&0&1&1
\end{pmatrix}
\end{equation*}
In these cases, a simple Sage computation tells us that 
\begin{align*}
K(G(\mathbb{F}_2^3,M)) &= (\Z/4\Z)^4 \oplus \Z/16\Z\\
K(G(\mathbb{F}_2^3,N)) &= (\Z/4\Z)^4 \oplus \Z/32\Z \oplus (\Z/3\Z)^6.
\end{align*}
Even though these two cases do not have the same number of invariant factors, they have the same number of \emph{Sylow}-2 invariant factors.
\end{example}

\begin{rem}
In linear algebra terms, \cref{prop:evenswitchtoodd} states that if one changes the parity of every entry of $L(G(\FF_2^r,M))$, the new matrix is of the same rank over $\FF_2$.
However, \cref{prop:evenswitchtoodd} is not quite a straightforward linear algebra result as it not true for integer matrices in general.
Indeed, even within the scope of this paper, the Laplacian of $K_4\cong G(\FF_2^2,({\bf e}_1,{\bf e}_2,{\bf e}_1+{\bf e}_2))$ shows that the parity assumption on $\mu_i$'s is necessary, and the ring theoretic proof is justified.
\end{rem}

\subsection{Generic Cases}

We now prove \cref{dMgeneric}, which we restate here:
\begin{theorem}

Suppose that $M$ is generic. Then the number of Sylow-2 cyclic factors of $K(G(\FF_2^r,M))$ is $2^{r-1} - 1$.
\label{thm:dMgenericc}
\end{theorem}
\begin{proof}
We want to find the dimension of 
\[
(K(G) \oplus \Z)\otimes(\Z/2\Z) \cong (\Z/2\Z)[x_1, \dots, x_r]/\left(x_1^2 - 1, \dots, x_r^2 - 1, n - \sum_{i = 1}^n\prod_{j= 1}^rx_j^{(v_i)_j} \right)
\]
as a vector space over $\Z/2\Z$.

Again making the change of variables $u_i := x_i - 1$ yields
\[
(K(G) \oplus \Z)\otimes(\Z/2\Z)\cong R:=(\Z/2\Z)[u_1, \dots, u_r]/\left(u_1^2, \dots, u_r^2, p(u_1,\dots, u_r):=n - \sum_{i = 1}^n\prod_{j = 1}^r(u_j + 1)^{(v_i)_j}\right).
\]
Since $\sum_{i=1}^n v_i \neq 0$, without loss of generality, $\sum_{i=1}^n (v_i)_r = 1$. As a result, the coefficient of $u_r$ in $p(u_1,\dots, u_r)$ is 1. $p(u_1,\dots, u_r)$ has trivial constant term, so we can rewrite $p(u_1, \dots, u_r) =   u_r \cdot f - g + (\text{non-squarefree terms})$, where $f,g$ are polynomials in $u_1,\ldots,u_{r-1}$ and $f$ has constant term $1$. Since $f$ has nonzero constant term and  $u_i^2 \equiv 0,\forall i$, we in fact have $f^2 \equiv 1$, so $f$ is invertible and $u_r \equiv fg$ in $R$. We can now construct a ring homomorphism
\[
T: R \to (\Z/2\Z)[u_1, \dots, u_{r-1}]/(u_1^2,\dots, u_{r-1}^2)
\]
by mapping $u_t \mapsto u_t$ for $t < r$ and $u_r \mapsto fg$.
$T$ is surjective for every monomial in $(\Z/2\Z)[u_1, \dots, u_{r-1}]$ is the image of the same monomial in $(\Z/2\Z)[u_1, \dots, u_r]$.

As a vector space over $\Z/2\Z$, the dimension of $(\Z/2\Z)[u_1, \dots, u_{r-1}]/(u_1^2,\dots, u_{r-1}^2)$ is $2^{r-1}$ as the standard monomials (with respect to any monomial order) are precisely the squarefree monomials in $u_1,\ldots, u_{r-1}$.
On the other hand, we claim that $\dim_{\Z/2\Z} R \le 2^{r-1}$ by showing that the standard monomials must be squarefree and $u_r$-free with respect to the lexicographic order, with $u_r$ being the largest variable: squarefree-ness is routine; for any monomial of the form $u_r\cdot u$ where $u$ is a squarefree monomial in $u_1,\ldots,u_{r-1}$, consider $(uf)p =(u_r\cdot u)f^2-ufg+uf({\rm non-squarefree\ terms})\in I(G)$, every term other than $u_r\cdot u$ is either not squarefree or does not contain $u_r$, so $u_r\cdot u$ is the leading term of [$(uf)p$ subtracting the non-squarefree terms], which is still in $I(G)$.
Thus, the map $T$ must be an isomorphism by comparing dimensions. Therefore, $(\Z \oplus K(G)) \otimes \Z/2\Z \cong (\Z/2\Z)^{2^{r-1}}$ as vector spaces, and $K(G)$ has $(2^{r-1} - 1)$ Sylow-$2$ cyclic factors.\end{proof}

\subsection{Non-generic Cases}

\noindent In the non-generic cases, the last generator of $I(G)$ no longer has a degree $1$ term, so we cannot construct the isomorphism in the proof above. Nevertheless, we in turn prove that \cref{dMgeneric} actually {\em characterizes} generic sets of generators.
The following theorem restates \cref{dMnon-generic} with a stronger bound when $n$ is not too big compared to $r$.

\begin{theorem} \label{thm:dMnon-genericc}
Suppose that $M$ is not generic. Then the number of Sylow-2 cyclic factors of $K(G(\FF_2^{r}, M))$ is at least $2^{r-1} - 1 + 2^{r-\lceil n/2\rceil}$; when $r<\lceil n/2\rceil$, the inequality simply states that the number is at least $2^{r-1}$.
\end{theorem}

For the rest of this section, we consider the purely linear algebraic point of view of the problem as in \cref{lem:even_equal_2rank}, although it would also be interesting to find a Gr\"{o}bner theoretic proof as in the generic case.
In particular, we view every matrix of interest as a matrix over $\FF_2$ and we need to prove that the rank of the Laplacian is at most $2^{r-1}-2^{r-\lceil n/2\rceil}$.

The simple but essential observation is that every Cayley graph of $\FF_2^r$ can be considered as the quotient of some hypercube.

\begin{definition}
Let $G=G(\FF_2^r,M)$ be a Cayley graph and $U\subset\FF_2^r$ be a subspace, such that $U\cap M=\emptyset$.
The {\em quotient Cayley graph} of $G$ by $U$, denoted by $G/U=G(\FF_2^r/U,M)$, is the graph with vertex set $\FF_2^r/U$ and multiedge set $\big\{(w+U, v_i+w+U)\big\}$ for coset $w+U$ and $v_i\in M$.

In other words, for any coset $X$ and $Y$, fix element ${\bf x}\in X$, then every edge between two vertices ${\bf x}\in X,{\bf y}\in Y$ in $G$ descends to an edge between $X,Y$ in $G/U$.
\end{definition}

\begin{lemma} \cite[Proposition~37]{IKKY} \label{lem:Cay_equal_HCq}
Let $M$ be the matrix whose columns are the generators of the Cayley graph.
Then $G(\FF_2^r,M)\cong Q_n/\ker M$ as graphs.
\end{lemma}

By definition, $M$ is non-generic if and only if the all one vector ${\bf 1}$ is contained in $\ker M$, and we can construct any non-generic Cayley graph of $\FF_2^r$ from $Q_n/\langle{\bf 1}\rangle$ by taking successive codimensional one quotients. This allows us to prove \cref{thm:dMnon-genericc} by induction.

\begin{proposition} \label{prop:2_rank_base}
Let $r\geq 2$ and $M_r = \begin{pmatrix}
{\bf e}_1 & \cdots & {\bf e}_r & {\bf e}_1+\ldots+{\bf e}_r
\end{pmatrix}.$
Then the rank of $L(G(\FF_2^r,M_r))$ is equal to $2^{r-1}-2^{\lfloor (r-1)/2\rfloor}$.
\end{proposition}

\begin{proof}
We prove by induction on $r$.
For $r=2$ and respectively $r=3$, the corresponding Cayley graph is $K_4$ and respectively $K_{4,4}$, their Laplacians' ranks are $1$ and $2$, respectively.

To ease the notation, we order the rows and columns of the Laplacian in the order of $w_0,\ldots,w_{2^r-1}\in\FF_2^r$.
Denote by $L$ the Laplacian of $Q_r$ {\em plus} the ($2^r\times 2^r$) identity matrix $I$, and denote by $\overline{L}$ and $\overline{I}$ the matrices obtained from $L$ and $I$ by reversing the ordering of the rows.
Note that since $L$ and $I$ are also symmetric with respect to the main skew diagonal, $\overline{L}$ and $\overline{I}$ can be obtained from $L$ and $I$ by reversing the ordering of their columns as well.
Now the Laplacian of $G(\FF_2^{r+2},M_{r+2})$ has the block structure
$
\begin{pmatrix}
L & I & I & \overline{I}\\
I & L & \overline{I} & I\\
I & \overline{I} & L & I\\
\overline{I} & I & I & L
\end{pmatrix}.
$

We then perform the following row and column operations in blocks.
Here $C\leftarrow C+\overline{C'}$ means we add the $(2^r-i+1)$-th column of the block $C'$ to the $i$-th column of $C$, the meanings of $R\leftarrow R+\overline{R'}, R\leftarrow \overline{R}, C\leftarrow\overline{C}$ are similar.

\begin{gather*}
\begin{pmatrix}
L & I & I & \overline{I}\\
I & L & \overline{I} & I\\
I & \overline{I} & L & I\\
\overline{I} & I & I & L
\end{pmatrix}
\xrightarrow[C_3\leftarrow C_3+C_2, C_4\leftarrow C_4+C_1]{R_3\leftarrow R_2+R_3, R_4\leftarrow R_4+R_1}
\begin{pmatrix}
L & I & 0 & L+\overline{I}\\
I & L & L+\overline{I} & 0\\
0 & L+\overline{I} & 0 & 0\\
L+\overline{I} & 0 & 0 & 0
\end{pmatrix}\\
\xrightarrow[C_1\leftarrow C_1+\overline{C_3},C_2\leftarrow C_2+C_3,C_4\leftarrow C_4+\overline{C_3}]{R_2\leftarrow R_2+\overline{R_1}}
\begin{pmatrix}
L & I & 0 & L+\overline{I}\\
0 & 0 & L+\overline{I} & 0\\
0 & L+\overline{I} & 0 & 0\\
L+\overline{I} & 0 & 0 & 0
\end{pmatrix}\\
\xrightarrow[C_2\leftarrow C_2+\overline{C_1}]{R_1\leftarrow R_1+\overline{R_3}, R_4\leftarrow R_4+\overline{R_3}}
\begin{pmatrix}
L & 0 & 0 & L+\overline{I}\\
0 & 0 & L+\overline{I} & 0\\
0 & L+\overline{I} & 0 & 0\\
L+\overline{I} & 0 & 0 & 0
\end{pmatrix}.
\end{gather*}

The matrix $L+\overline{I}$ is the Laplacian of $G(\FF_2^r,M_r)$, hence the second and third block of columns each has rank $2^{r-1}-2^{\lfloor (r-1)/2\rfloor}$ by induction hypothesis.
For the first and forth block of columns, we have the $2^{r+1}\times 2^{r+1}$ nonzero submatrix $
\begin{pmatrix}
L &  L+\overline{I}\\
L+\overline{I} & 0
\end{pmatrix},
$
in which we can perform further row and column operations:
$
\begin{pmatrix}
L &  L+\overline{I}\\
L+\overline{I} & 0
\end{pmatrix}
\xrightarrow[C_2\leftarrow C_2+C_1]{R_2\leftarrow R_2+R_1}
\begin{pmatrix}
L &  \overline{I}\\
\overline{I} & L
\end{pmatrix}
\xrightarrow[C_1\leftarrow\overline{C_1}]{R_1\leftarrow \overline{R_1}}
\begin{pmatrix}
L &  I\\
I & L
\end{pmatrix}.
$
This is the Laplacian of $Q_{r+1}$, and is of rank $2^r$ by \cref{dMgeneric}.

Summing the ranks of all the blocks here, we can see that $L(G(\FF_2^{r+2},M_{r+2}))$ is of rank $2(2^{r-1}-2^{\lfloor (r-1)/2\rfloor})+2^r=2^{r+1}-2^{\lfloor (r+1)/2\rfloor}$.
\end{proof}

\begin{rem}
The number $2^{r-1}-2^{\lfloor (r-1)/2\rfloor}$ also occurs as the number of $\Z/2\Z$ summands in the elementary divisor decomposition of $K(Q_{r+1})$ \cite[Theorem~1.3]{bai}.
We are not sure if that is a coincident, or a special case of some deeper connection between the elementary divisors of hypercubes and the number of $2$-cyclic factors of other Cayley graphs.
\end{rem}

\begin{proof} [Proof of \cref{thm:dMnon-genericc}]
As discussed after the statement of \cref{lem:Cay_equal_HCq}, we consider a non-generic Cayley graph as the quotient graph $Q_n/U$ for some subspace $U\subset\FF_2^n$ that contains ${\bf 1}$, and prove the theorem by induction on $\dim U$.
When $\dim U=1$, we may assume $M=M_r$ by a left multiplication of suitable element of $GL_r(\FF_2)$, and the base case of the induction is \cref{prop:2_rank_base} as $r-\lceil (r+1)/2\rceil = \lfloor (r-1)/2\rfloor$.

Suppose $\dim U>1$.
Pick a codimension one subspace $U'$ of $U$ that contains ${\bf 1}$, and pick an arbitrary ${\bf u}\not\in U'$.
Since ${\bf x},{\bf y}\in\FF_2^n$ descend to the same vertex in $Q_n/U$ if and only if ${\bf x}+U',{\bf y}+U'\in\FF_2^n/U'$ descend to the same vertex in $(Q_n/U')/\langle {\bf u}+U'\rangle$, we can see that $Q_n/U\cong (Q_n/U')/\langle {\bf u}+U'\rangle$.
By induction hypothesis, the rank of the Laplacian of $Q_n/U'$ is equal to $2^r-m$ for some $m\geq 2^{r+1-\lceil n/2\rceil}$.

Choose an arbitrary $(2^{r-1}-\lfloor(m-1)/2\rfloor)\times (2^{r-1}-\lfloor(m-1)/2\rfloor)$-submatrix $M$ of $L(Q_n/U)$, whose rows and columns are indexed by $R,C\subset\FF_2^n/U$, respectively.
Pick any $(2^{r-1}-\lfloor(m-1)/2\rfloor)$-subset $R_1\subset\FF_2^n/U'$ whose image is $R$ under the projection map $\FF_2^n/U'\rightarrow\FF_2^n/U$, and let $R_2=R_1+({\bf u}+U')$ be its complement in the preimage of $R$, order both subsets according to the ordering of their images in $R$; define and order $C_1$ and $C_2$ for $C$ similarly.

Consider the $(2^r-2\lfloor(m-1)/2\rfloor)\times (2^r-2\lfloor(m-1)/2\rfloor)$-submatrix $M'$ of $L(Q_n/U')$, whose rows and columns are indexed $R_1\sqcup R_2$ and $C_1\sqcup C_2$, respectively.
The (ordered) subsets $R_1$ and $R_2$ (respectively, $C_1$ and $C_2$) differ coordinate-wise by the translation of ${\bf u}+U'$, which implies that $M'$ has a block structure 
$
\begin{pmatrix}
A & B\\
B & A
\end{pmatrix}.
$
Moreover, $M$ itself is equal to $A+B$ by the definition of quotient.
Since $2^r-2\lfloor(m-1)/2\rfloor>2^r-m$, $M'$ is singular.
Pick an arbitrary nonzero element $\begin{pmatrix} {\bf v}_1\\ {\bf v}_2 \end{pmatrix}$ from the kernel of $M'$, i.e., 
\begin{equation} \label{eq:kernel_element}
    A{\bf v}_1+B{\bf v}_2=A{\bf v}_2+B{\bf v}_1={\bf 0}.
\end{equation}
If ${\bf v}_1\neq {\bf v}_2$, then we have $M({\bf v}_1+{\bf v}_2)=(A{\bf v}_1+B{\bf v}_2)+(A{\bf v}_2+B{\bf v}_1)={\bf 0}$, so $M$ is singular as ${\bf v}_1+{\bf v}_2\neq{\bf 0}$ is in its kernel.
If ${\bf v}_1={\bf v}_2\neq{\bf 0}$, then \cref{eq:kernel_element} implies $A{\bf v}_1+B{\bf v}_1={\bf 0}$, which shows that ${\bf v}_1$ is in the kernel of $M$ and again shows that $M$ is singular.

Since $M$ is arbitrary, the rank of $L(Q_n/U)$ is at most $2^{r-1}-\lfloor(m-1)/2\rfloor-1=2^{r-1}-\lfloor(m+1)/2\rfloor$.
If $r+1>\lceil n/2\rceil$, then $2^{r+1-\lceil n/2\rceil}\geq 2$ is an even integer, and $\lfloor(m+1)/2\rfloor \geq\lfloor(2^{r+1-\lceil n/2\rceil}+1)/2\rfloor=2^{r-\lceil n/2\rceil}$.
Otherwise, we at least have $m\geq 1$, and $\lfloor(m+1)/2\rfloor\geq 1>2^{r-\lceil n/2\rceil}$ for $r<\lceil n/2\rceil$.
In both cases, the inequality in the theorem statement is verified.
\end{proof}

We conclude this section with some implications in the other setting that was mentioned in \cref{importantremark}.
A {\em signed graph} is an ordinary graph with some edges labeled as negative. The Laplacian of a signed (simple) graph can be obtained from the ordinary Laplacian by flipping the signs of the off-diagonal entries corresponding to the negative edges; the critical group of the signed graph is the cokernel of the signed Laplacian \cite{reinersign}.
In \cite{IKKY}, Iga et al. studied a class of signed graphs known as {\em Adinkras}, which are graphical gadgets that encode special supersymmetry and Clifford algebras (much like Cayley graphs encode finite groups); it is known that the underlying graph of any Adinkra must be a Cayley graph of some $\FF_2^r$.
In that work, the authors were able to determine the Sylow-$p$ subgroup of the critical group of an Adinkra for every odd prime $p$, and the Sylow-$2$ subgroup when the rank of the signed Laplacian over $\Z/2\Z$ is exactly $2^{r-1}$.
Since the rank over $\Z/2\Z$ is independent of the signs of the entries, the results of this section imply that their results on Sylow-$2$ subgroups apply precisely when the underlying Cayley graph is generic; see \cite[Section~4]{Yuen_Adinkra} for further discussion on the role of genericity in the theory of Adinkras.
\section{Bounding the Largest Cyclic Factor}\label{section: boundCyclicFactorSection}

We are interested in the largest powers of $2$ dividing the cyclic factors $c_i$ (i.e., their $2$-adic valuations), since all of the $p \neq 2$ information is determined by \cref{prop:pnot2}.
Anzis and Prasad \cite{anzisprasad} proved that for $G = Q_n$, the largest cyclic factor $c_1(Q_n)$ must divide $2^n\text{lcm}(1,\dots,n)$. As an immediate corollary, the largest 2-cyclic factor is bounded by $2^{\lfloor \log_2 n\rfloor + n}$. In this section, we adopt and improve their approach, and prove a tighter bound for all Cayley graphs of $\mathbb{F}_2^r$.

\begin{theorem}
\label{thm: largeCycFactorTheorem}
Let $G = G(\mathbb{F}_2^r, M)$ be a Cayley graph, and let $(\lambda_1 = 0), \lambda_2, \dots, \lambda_{2^r}$ be the eigenvalues of the $L(G)$.
 Then 
 \[
 c_1(G) \mid 2^{r-2} \on{lcm}\left(\lambda_{i}: i \ge 2\right).
 \]
\end{theorem}
For $G = Q_n$, the eigenvalues of $L(G)$ are of the form $2k$ for $0 \le k \leq n$ with multiplicity $\binom{n}{k}$. In this case our bound is $2^{n-2}\on{lcm}(2,\dots, 2n) = 2^{n-1}\on{lcm}(1,\dots, n)$, which improves the bound in \cite{anzisprasad} by a factor of $2$. This also improves the general bound in \cite[Corollary~3]{rushanan}, which states that $\on{lcm}(\lambda_{i})$ divides $c_1(G)$, and $c_1(G)$ in turn divides the product of all {\em distinct} nonzero eigenvalues $\prod_{\lambda\in
\{\lambda_{i}: \lambda_i\neq 0\}}\lambda$ for any integer matrix with integer eigenvalues.

First, using the ring-theoretic description in \cref{cokerIso}, we have the following lemma.

\begin{lemma} 
\label{lemma:LinearAlgebraCyclicOrder}
Suppose a nonzero $\overline{p(x_1, \dots, x_r)} \in \Z[x_1, \dots, x_r]/I(G) \simeq \Z \oplus K(G)$ has finite additive order and let $\mathbf{w}$ be the vector in $\Z^{2^r}$ corresponding to $p(x_1, \dots, x_r)$ under the isomorphism 
$$
\Z^{2^r} \cong \Z[x_1, \dots, x_r]/(x_1^2 - 1, \dots, x_r^2 - 1).
$$
Then the additive order of $\overline{p(x_1, \dots, x_r)}$ is
\[
|\mathbf{w}| := \min\{ C \in \Z_{>0}: C\mathbf{w} \in \Img L(G)\}.
\]
\end{lemma}
\begin{proof}
This follows by the definition of cokernel and considering the cokernel as a $\Z$-module.
\end{proof}

From now on, we denote $\on{ord}(p(x_1, \dots, x_r))$ as the additive order of $\overline{p(x_1, \dots, x_r)}$ in $\Z[x_1, \dots, x_r]/I(G)$.

\begin{lemma}
\label{maxCyclicLemma}
 $c_1(G) = \on{lcm}\left(\on{ord}(x_k - 1): 1\le k \le r\right)$.
\end{lemma}
\begin{proof}

Consider the map $\pi:\Z[x_1,\dots,x_r]/I(G)\cong \Z\oplus K(G) \to \Z$ that sends all $x_i\mapsto 1$. This is a well-defined map because $(x_i=1)$ is a common zero for all polynomials in $I(G)$. Since $\pi$ is surjective, $(\Z[x_1,\dots,x_r]/I(G))/\on{ker}(\pi)\cong\Z$, which forces $\on{ker}(\pi)$ to be the torsion part $K(G)$, and an element has finite additive order in the ring $\Z[x_1,\dots,x_r]/I(G)$ iff it lies in the kernel of $\pi$. Furthermore, any polynomial with finite additive order is a linear combination of $x_I - 1$, where $x_I$ denotes a squarefree monomial: we can write any polynomial in the ring as a sum $\sum_I c_Ix_I$ of squarefree monomials since $x_i^2=1$, and if the polynomial is in the kernel of $\pi$, then $\sum_I c_I=0$ thus we can rewrite it as $\sum_I c_I(x_I-1)$.

Now let $C = \on{lcm}\left(\on{ord}(x_k - 1): 1\le k \le r\right)$, and let $x_I$ be any squarefree monomial. We wish to show that $C(x_I -1) \equiv 0$, i.e., $C(x_I-1) \in I(G)$. Indeed, suppose $x_j \mid x_I$. Then we have $C(x_I - 1) = C(x_j - 1)\cdot \frac{x_I}{x_j} + C\left(\frac{x_I}{x_j} - 1\right)$, which we can reduce inductively to a linear combination of $x_I - 1$ with $\deg x_I = 1$. This shows that the largest cyclic factor is equal to the lcm of the orders taken over all $x_i - 1$, which is the desired result.
\end{proof}
\begin{rem}
This lemma can actually be slightly generalized. Namely, let $v_1, \dots, v_n$ be any generating set of $\FF_2^r$. Then the largest possible additive order equals the lcm of $\on{ord}\left(\prod_j x_j^{(v_i)_j} - 1\right)$ over all $1\leq i\leq n$. Anzis and Prasad's original argument shows that for the hypercube, we can take any $x_k - 1$ for $1 \le k \le n$. The argument relies on showing that $x_i - 1$'s have the same additive order, which follows from symmetry of the \emph{generators} under permutation. However, this is no longer the case for general Cayley graphs.
\end{rem}

Denote the lowest common denominator (lcd) of a set of fractions $A\subset\Q$ as the smallest positive integer $\on{lcd}(A)$ such that $\on{lcd}(A)\cdot a\in \Z$ for all $a\in A$; lcd obviously exists when $A$ is finite. The next lemma gives an explicit way for calculating $\on{ord}(x_j-1)$, \cref{thm: largeCycFactorTheorem} follows from it.

\begin{lemma}
\label{orderofmonomial}
The order of $x_j - 1$ in $K(G)$ is equal to the lowest common denominator of the following rational numbers:
\[
\frac{1}{2^{r-2}}\sum_{\substack{u\cdot v=1 \\ u_j = 1}}\frac{1}{\lambda_u},\quad \text{for }v \in \FF_2^r.
\]
\end{lemma}
\begin{proof}
Following from \cref{lemma:LinearAlgebraCyclicOrder}, the order of $x_j-1$ is the smallest positive integer $C$ such that there is a solution $\textbf{v} \in \Z^{2^r}$ to $L(G)\textbf{v} = C\textbf{w}$ for $\mathbf{w} = x_j - 1 \leftrightarrow (-1,0, \dots,1, \dots, 0)=e_{w_{i}}-e_{w_0}$. We first solve the equation $L(G)\mathbf{v} = \mathbf{w}$ by expressing $\mathbf{w}$ as a linear combination of the eigenbasis of $L(G)$. By \cref{changeofbasis}, we change to the eigenbasis and get
\[
\mathbf{w} = x_j - 1 \longmapsto \frac{1}{2^r}\left(\sum_{v \in \FF_2^r}\left((-1)^{v \cdot e_j} - (-1)^{v \cdot 0}\right)f_v\right) = -\frac{1}{2^{r-1}}\sum_{u_j = 1}f_u.
\]
Since $f_u$ is an eigenbasis, the equation $L(G)\mathbf{v} = \mathbf{w}$ has the following particular solution $\mathbf{v}=\mathbf{v}_0$: 
\[
\mathbf{v}_0 = -\frac{1}{2^{r-1}}\sum_{u_j = 1}\frac{1}{\lambda_u}f_u = -\frac{1}{2^{r-1}}\sum_{v \in \FF_2^r}\sum_{u_j = 1}\frac{(-1)^{u \cdot v}}{\lambda_u}\cdot e_v.
\]

Since the graph $G$ is connected, all solutions to $L(G)\mathbf{v} = \mathbf{w}$ are of the form $\mathbf{v}_0 + k\mathbf{1}^T$ where $\mathbf{1}=(1,\dots, 1)$ lies in the kernel. Therefore, $C$ is the smallest integer such that $C\mathbf{v}_0+k\mathbf{1}$ is an integer vector for some $k\in\R$. Equivalently speaking, $C$ is the lowest common denominator of pairwise-differences of $\mathbf{v}_0$'s entries given by
\[(\mathbf{v}_0)_{v} - (\mathbf{v}_0)_{\vec{0}}=\frac{1}{2^{r-1}}\sum_{u_j=1}\frac{(-1)^{u \cdot v}-1}{\lambda_u} = -\frac{1}{2^{r-2}}\sum_{\substack{u\cdot v=1 \\ u_j=1}}\frac{1}{\lambda_u},\quad \text{for all } v\in \FF_2^r.\]
\end{proof}

\begin{proof}[Proof of \cref{thm: largeCycFactorTheorem}]
This follows immediately from \cref{maxCyclicLemma} and \cref{orderofmonomial}.
\end{proof}

\begin{corollary}
\label{Syl2topcyclic}

The largest $2$-cyclic factor $\Z / 2^e \Z$ of $K(G)$ for the Cayley Graph $G=G(\FF_2^r,M)$ has bound
\[
e = v_2(c_1(G)) \leq \lfloor\log_2(n) \rfloor + r - 1.
\]
\end{corollary}
\begin{proof}
From \cref{thm: largeCycFactorTheorem} we know that $c_1(G)$ must divide $2^{r-2} \text{lcm}\left(\lambda_i: i \ge 2\right)$. Therefore, $v_2(c_1(G))$ must be capped by the $2$-adic valuation of the right hand side, which is
\[e = v_2(c_1(G)) \leq \max_{i\geq 2}(v_2(\lambda_i)) + r - 2.\]
However, note that $2 \leq \lambda_i \leq 2n$ for $i \ge 2$, so all $v_2(\lambda_i)$ is upper-bounded by $\lfloor\log_2(2n) \rfloor = \lfloor\log_2(n) \rfloor +1$.
\end{proof}

What is especially nice about this improvement is that it is asymptotically tight:  We can see in \cref{corollary:asymptoticsharp} that this upper bound on the $2$-adic valuation is achieved for all hypercubes of dimension $2^k$ or $2^k + 1$. This is an immediate consequence of the main result of the next section, which completely determines the top cyclic factor of the hypercube. We now discuss some preliminary results that lead to this sharper $Q_n$ result.

\begin{rem}
Since we only care about the Sylow-$2$ factor in these maximal orders, it actually suffices to find the minimal $C$ such that for any $v \in \FF_2^r$,
\[
\frac{C}{2^{r-2}}\sum_{\substack{u \cdot v = 1\\u_j = 1}}\frac{1}{\lambda_u} \in \Z_{(2)}
\] where $\Z_{(2)}$ is the localization of the integers away from the prime ideal $(2)$. This way, we do not actually care about odd denominators.
\end{rem}

Although \cref{orderofmonomial} gives us a closed form to calculate the order of $x_j-1$, the alternating sums in \cref{orderofmonomial} are too complicated to help us prove useful properties of these numbers or perform calculations. Our solution is to transform these sums into a set of simpler numbers that have the same lowest common denominator. In particular, we use the following claim.

\begin{claim}
Given two (finite) sets of rational numbers $A$ and $B$, if $\on{span}_\Z(A)=\on{span}_\Z(B)$, then $A$ and $B$ have the same lowest common denominator.
\end{claim}

The next change-of-basis lemma simplifies the complicated sums in \cref{orderofmonomial}, thus helps us to prove a simpler closed form for $c_1(G)$ by the end of this section.

\begin{lemma}
\label{lem:spanlemma}
In the algebra $\Z\left[a_S:\emptyset\neq S\subset[r]\right]$, we have the following equality:
$$\on{span}_\Z\left\{\sum_{|S\cap T|\text{ odd}}a_S\middle|\emptyset\neq T\subset[r]\right\}=\on{span}_\Z\left\{2^{|T|-1}\sum_{T\subset S}a_S\middle|\emptyset\neq T\subset [r]\right\}.$$
\end{lemma}

\begin{proof}
Denote the sums on the left hand side as $x_T=\sum_{|S\cap T|\text{ odd}}a_S$ and the sums on the right hand side as $y_T=2^{|T|-1}\sum_{T\subset S}a_S$. We will first show that every $x_T$ can be expressed as an integer linear combination of the $y_W$'s.

\begin{claim}
    \[x_T=\sum_{\emptyset\neq W\subset T}(-1)^{|W|-1}y_W.\]
\end{claim}

We can verify this by computation.
\begin{equation}\label{eq:claim1}
    \sum_{\emptyset\neq W\subset T}(-1)^{|W|-1}y_W = \sum_{\emptyset\neq W\subset T}\left((-2)^{|W|-1}\cdot\sum_{W\subset S}a_S\right) = \sum_{S\subset [r]}\left(a_S\cdot \sum_{\emptyset\neq W\subset S\cap T}(-2)^{|W|-1}\right).
\end{equation}
Now we investigate the rightmost sum in \cref{eq:claim1}. If $|S\cap T|=k$, then
\[\sum_{\emptyset\neq W\subset S\cap T}(-2)^{|W|-1}=\sum_{1\leq p\leq k}(-2)^{p-1}\binom{k}{p}=-\frac{1}{2}\cdot\left((-2+1)^k-1\right)=\begin{cases}0, &\text{if $k$ is even},\\1,&\text{if $k$ is odd.}\end{cases}\]
Plug it back into \cref{eq:claim1} and we get exactly the equation in the claim.

Now if we look at the change-of-basis matrix between $x_T$'s and $y_T$'s, it is an upper-triangular matrix with all $\pm 1$'s on the main diagonal (if we order variables by inclusion). Therefore the matrix is invertible over $\Z$.
\end{proof}

With the help of \cref{lem:spanlemma}, we can rewrite the equation for $\on{ord}(x_j-1)$ in \cref{orderofmonomial} into the following corollary.

\begin{corollary}
\label{orderofmonomialnew}
The order of $x_j - 1$ in $K(G)$ is equal to the lowest common denominator of the following rational numbers:
\[
\frac{1}{2^{r-|S|}}\sum_{\substack{u_s=1\\\forall s\in S}}\frac{1}{\lambda_u},\quad \text{for }j\in S \subset [r]\text{ and }|S|\geq 2\quad\text{ and }\quad \frac{1}{2^{r-2}}\sum_{u_j=1}\frac{1}{\lambda_u}.
\]
\end{corollary}

\begin{proof}
Recall that we define $w_S\in\FF_2^r$ to be the $0$-$1$ vector with entry $1$ at indices $s\in S$ and $0$ elsewhere. In \cref{lem:spanlemma}, we assign $a_S=1/\lambda_{w_S}$ if $j\in S$ and $a_S=0$ if $j\notin S$. Then both sides in \cref{lem:spanlemma} have the same lowest common denominators. The left hand side is the order of $x_j-1$ by \cref{orderofmonomial}. The right hand side is the lowest common denominator of the following numbers:
\[
\frac{1}{2^{r-|T|-1}}\sum_{\substack{u_T=1\\u_j=1}}\frac{1}{\lambda_u},\quad \text{for }T\not\ni j\quad\text{ and }\quad\frac{1}{2^{r-|T|-1}}\sum_{\substack{u_T=1}}\frac{1}{\lambda_u},\quad \text{for }T\ni j.
\]
Let $T' = T\cup\{j\}$. Then we can rewrite the numbers above as:
\[
\frac{1}{2^{r-|T'|}}\sum_{\substack{u_{T'}=1}}\frac{1}{\lambda_u},\quad \text{for }j\in T'\text{ and }|T'|\geq 2\quad\text{ and }\quad\frac{2}{2^{r-|T'|}}\sum_{\substack{u_{T'}=1}}\frac{1}{\lambda_u},\quad \text{for }j\in T'.
\]
The left hand side covers the right hand side except when $|T'|=1$, or $T'=\{j\}$. The left hand side corresponds to LHS in the corollary, and the case $T'=\{j\}$ corresponds to RHS in the corollary.
\end{proof}

Finally, we can combine the orders for all $x_j-1$ and prove the following closed form of the largest cyclic factor $c_1(G)$ of $K(G)$.

\begin{corollary} \label{corollary: topcyclic}
For Cayley Graph $G=G(\FF_2^r,M)$, $c_1(G)$ is equal to the lowest common denominator of the following rational numbers:
\[
\frac{1}{2^{r-|S|}}\sum_{\substack{u_s=1\\\forall s\in S}}\frac{1}{\lambda_u},\quad \text{for }S \subset [r]\text{ and }|S|\geq 2\quad\text{ and }\quad \frac{1}{2^{r-2}}\sum_{u_j=1}\frac{1}{\lambda_u},\quad \text{for }1\leq j\leq r.
\]
\end{corollary}

In fact, this is equivalent to saying that if we draw a $r$-dimensional hypercube, and put $1/\lambda_u$ on each vertex $u\neq\mathbf{0}$, then $c_1(G)$ is equal to the lowest common denominator of the arithmetic means of numbers on a face for all faces with codimension $\geq 2$ that do not pass through the origin.
\section{The Top Cyclic Factors for Hypercube Graphs}
\label{topcyclicQnsection}
In this section, we specialize to the case $G = Q_n$. We will use the techniques developed in the past section to prove the following theorems.

\begin{theorem}
\label{thm:topcyclicQn}
For $n\geq 2$, let $c_1(Q_n)$ be the size of the largest cyclic factor in $K(Q_n)$. Then,
\[
v_2(c_1(Q_n))=\max\left\{\max_{x<n}\{v_2(x)+x\}, v_2(n)+n-1\right\}.
\]
\end{theorem}

\begin{theorem}
\label{thm:2ndtopcyclicQn}
For $n\geq 3$, the $2^\text{nd}$ to the $(n-1)^\text{th}$ largest $2$-cyclic factor in $K(Q_n)$ all have the same size $v_2(c_2(Q_n))$. Moreover,
$$v_2(c_2(Q_n)) =v_2(c_3(Q_n)) =\cdots= v_2(c_{n-1}(Q_n)) =\max_{x<n}\{v_2(x)+x\}.$$
\end{theorem}

Specifically, when $n=2^k$ or $2^k+1$, \cref{thm:topcyclicQn,thm:2ndtopcyclicQn} implies the following corollary, which shows that the bound $c_1(G)\leq r+\lfloor\log_2 n\rfloor-1$ in \cref{Syl2topcyclic} is asymptotically sharp over all Cayley graphs of $\FF_2^r$.
\begin{corollary}
\label{corollary:asymptoticsharp}
\begin{enumerate}[label = (\alph*)]
    \item For $n = 2^k -1$, the top $(n-1)$ Sylow-$2$ cyclic factors have exponent $2^k - 1$.
    \item For $n = 2^k$, the top Sylow-$2$ factor is $2^k + k - 1$. The $2$nd through $(n-1)$st Sylow-$2$ factors are $2^{k} - 1$.
    \item For $n = 2^k + 1$, the top $(n-1)$ Sylow-$2$ cyclic factors are $2^k + k$.
\end{enumerate}
\end{corollary}

\begin{rem}
We note that one only needs to inspect at most $\log_2 n$ many candidates for $x$ when finding $\max_{x<n}\{v_2(x)+x\}$.
Rewriting $\max_{x<n}\{v_2(x)+x\}=\max_{k\leq\log_2 n}\{\max_{v_2(x)\geq k, x<n}\{k+x\}\}$, for every $k$, we may assume $x$ is equal to $n-1$ setting the last $k$ digits in binary representation zeros, as it is the largest possible $x<n$ to have $v_2(x)\geq k$.

In particular, writing $n-1$ in binary, one only needs to compare the values of $v_2(x)+x$ for $x$'s obtained from $n-1$ (in binary representation) by successively changing the last nonzero digit to zero.
For example, if $n=54$ and $n-1=53=(110101)_2$, we only need to inspect $x=(110101)_2,(110100)_2,(110000)_2,(100000)_2=53,52,48,32$.
\end{rem}

In order to prove \cref{thm:topcyclicQn}, we first translate it into a question about binomial coefficients.

\begin{lemma} \label{lem:binomial}
$c_1(Q_n)$ is the the lowest common denominator of the following rational numbers:
\[
F_n(a):=\frac{1}{2^{n-a}}\sum_{i=0}^{n-a}\frac{\binom{n-a}{i}}{2(a+i)},\quad \text{for }2\leq a\leq n\quad\text{ and }\quad F_n(1):=\frac{1}{2^{n-2}}\sum_{i=0}^{n-1}\frac{\binom{n-1}{i}}{2(1+i)}.\]
\end{lemma}

\begin{proof}
This follows from \cref{corollary: topcyclic} and the fact that when $G=Q_n$, $\lambda_u=2w(u)$ (recall $w(u)$ is the number of $1$s in the vector $u\in\FF_2^r$). In particular, there are $\binom{n-a}{i}$ ways to choose $i$ 1's in a size $n-a$ vector.
\end{proof}

Notice that there are binomial coefficients in the expressions in the lemma above. Therefore, we will make extensive use of a classical theorem of Kummer's that computes the $2$-adic valuation of binomial coefficients in terms of their binary expansions.
\begin{theorem} (Kummer's Theorem)
For any non-negative integers $a\geq b$,
$$v_2\left(\binom{a}{b}\right)=\text{number of carries when adding $a-b$ to $b$ in binary}.$$
\end{theorem}
\begin{example}
\[
\begin{tikzpicture}[
    row 1/.style={font=\textsl,font=\scriptsize,black!85, anchor=west,
        inner sep=1.5pt},
    every node/.style={column sep=.5mm,row sep=.2mm}]
    \matrix (m) [matrix of math nodes,
        nodes in empty cells,
    ] 
    {
        &   &   & 1 & 1 & 1 & 1 & 1 & 1 &   &                   \\
        & 1 & 0 & 0 & 1 & 1 & 1 & 1 & 0 & 1 &[10mm]        317\\
    +   & 0 & 1 & 0 & 1 & 1 & 0 & 1 & 1 & 1 &        183\\ 
        & 1 & 1 & 1 & 1 & 1 & 0 & 1 & 0 & 0 &        500 \\                                                  
    };

    \draw[-,color=black,semithick] (m-3-2.south west) -- (m-3-10.south east);
\end{tikzpicture}
\]
Therefore, $v_2(\binom{500}{317})=6$, since there are $6$ carries.
\end{example}

The next lemma shows that the $2$-adic valuations of each sum in \cref{lem:binomial} is equal to the $2$-adic valuation of a single term within the sum.
\begin{lemma} \label{binomial2}
For any $p\geq 1$, $q\geq 0$, assume $u$ is the unique element in the interval $[p,p+q]$ that maximizes $v_2(u)$, then
$$v_2\left(\sum_{i=0}^{q}\frac{\binom{q}{i}}{p+i}\right)=v_2\left(\frac{\binom{q}{u-p}}{u}\right).$$
\end{lemma}
\begin{proof}
First we claim that for any $c\in[p,p+q]$,
\begin{equation}\label{eq:binomv2}
    v_2\left(\frac{\binom{q}{c-p}}{c}\right)\ge v_2\left(\frac{\binom{q}{u-p}}{u}\right).
\end{equation}
Denote $k:=v_2(u)$, then $q<2^{k+1}$. Now we compare $v_2\left(\binom{q}{c-p}\right)$ and $v_2\left(\binom{q}{u-p}\right)$, which by Kummer's Theorem is the number of carries in the sums $(c-p)+(p+q-c)$ and $(u-p)+(p+q-u)$, respectively. Notice that the binary form of $c-p$ and $u-p$ are the same in the last $v_2(c)$ bits, therefore both cases have the same number of carries in the last $v_2(c)$ bits. In the remaining $k-v_2(c)$ bits, the case $u-p$ has at most $k-v_2(c)$ carries and the case $c-p$ has at least $0$ carries. This worst scenario exactly results in equality.

For example, when $p=134$, $q=101$, we have $u=192$ and $k=6$. Now we analyze the case when $c=168$, $v_2(c)=3$. In the two vertical additions, last $v_2(c)=3$ bits (indicated by the red box) are identical, and therefore the first $3$ carries (indicated by the blue box) are identical. The number of carries differ in at most $3$, which is the same as $k-v_2(c)$.
\[
\begin{tikzpicture}[
    row 1/.style={font=\textsl,font=\scriptsize,black!85, anchor=west, inner xsep=1.5pt, inner ysep=3pt},
    every node/.style={column sep=.5mm,row sep=.2mm}
    ]
    \matrix (m) [matrix of math nodes,
        nodes in empty cells,
    ] 
    {
         & 1 & 1 & 1 &   & 1 &   &   &                   \\
        & 0 & 1 & 1 & 1 & 0 & 1 & 0 &[10mm]        u-p=58\\
    +    & 0 & 1 & 0 & 1 & 0 & 1 & 1 &        43\\ 
         & 1 & 1 & 0 & 0 & 1 & 0 & 1 &        q=101 \\                                                  
    };

    \draw[-,color=black,semithick] (m-3-2.south west) -- (m-3-8.south east);
    \draw[dashed,color=red,semithick] (m-4-6.south west) -- (m-4-8.south east) -- (m-2-8.north east) -- (m-2-6.north west) -- cycle;
    \draw[dashed,color=blue,semithick] (m-1-5.south west) -- (m-1-7.south east) -- (m-1-7.north east) -- (m-1-5.north west) -- cycle;
\end{tikzpicture}
\hspace{2cm}
\begin{tikzpicture}[
    row 1/.style={font=\textsl,font=\scriptsize,black!85, anchor=west,
        inner xsep=1.5pt, inner ysep=3pt},
    every node/.style={column sep=.5mm,row sep=.2mm}]
    \matrix (m) [matrix of math nodes,
        nodes in empty cells,
    ] 
    {
         &   &   &   &   & 1 &   &   &                   \\
        & 0 & 1 & 0 & 0 & 0 & 1 & 0 &[10mm]        c-p=34\\
    +    & 1 & 0 & 0 & 0 & 0 & 1 & 1 &        67\\ 
         & 1 & 1 & 0 & 0 & 1 & 0 & 1 &  q=101 \\                                                  
    };

    \draw[-,color=black,semithick] (m-3-2.south west) -- (m-3-8.south east);
    \draw[dashed,color=red,semithick] (m-4-6.south west) -- (m-4-8.south east) -- (m-2-8.north east) -- (m-2-6.north west) -- cycle;
    \draw[dashed,color=blue,semithick] (m-1-5.south west) -- (m-1-7.south east) -- (m-1-7.north east) -- (m-1-5.north west) -- cycle;
\end{tikzpicture}
\]
Finally, we show that the equality in \cref{eq:binomv2} will be achieved an odd numbers of times. According to the analysis above, equality occurs when $c-p$ does not carry in the highest $k-v_2(c)$ bits. However, we can negate the top bit of $c-p$. We reconsider the example above where $c=168$. In such case, we can switch the top bit (indicated by the orange box) of $c=168$ to get $c'=c+2^k=232$, such that $v_2(\binom{101}{34})=v_2(\binom{101}{98})$.
\[
\begin{tikzpicture}[
    row 1/.style={font=\textsl,font=\scriptsize,black!85, anchor=west,
        inner sep=1.5pt},
    every node/.style={column sep=.5mm,row sep=.2mm}]
    \matrix (m) [matrix of math nodes,
        nodes in empty cells,
    ] 
    {
         &   &   &   &   & 1 &   &   &                   \\
        & 0 & 1 & 0 & 0 & 0 & 1 & 0 &[10mm]        c-p=34\\
    +    & 1 & 0 & 0 & 0 & 0 & 1 & 1 &        67\\ 
         & 1 & 1 & 0 & 0 & 1 & 0 & 1 &  q=101 \\                                                  
    };

    \draw[-,color=black,semithick] (m-3-2.south west) -- (m-3-8.south east);
    \draw[dashed,color=orange,semithick] (m-3-2.south west) -- (m-3-2.south east) -- (m-2-2.north east) -- (m-2-2.north west) -- cycle;
\end{tikzpicture}
\hspace{2cm}
\begin{tikzpicture}[
    row 1/.style={font=\textsl,font=\scriptsize,black!85, anchor=west,
        inner sep=1.5pt},
    every node/.style={column sep=.5mm,row sep=.2mm}]
    \matrix (m) [matrix of math nodes,
        nodes in empty cells,
    ] 
    {
         &   &   &   &   & 1 &   &   &  \\
         & 1 & 1 & 0 & 0 & 0 & 1 & 0 &[10mm] c'-p=98\\
    +    & 0 & 0 & 0 & 0 & 0 & 1 & 1 & 3\\ 
         & 1 & 1 & 0 & 0 & 1 & 0 & 1 & q=101 \\                                                  
    };

    \draw[-,color=black,semithick] (m-3-2.south west) -- (m-3-8.south east);
    \draw[dashed,color=orange,semithick] (m-3-2.south west) -- (m-3-2.south east) -- (m-2-2.north east) -- (m-2-2.north west) -- cycle;
\end{tikzpicture}
\]
Therefore, there is a pairing of the equality cases $c+2^k-p$ and $c-p$ when $c\neq u$. This ends the proof.
\end{proof}

Now we have all the tools to calculate the $2$-adic valuation of $c_1(Q_n)$.

\begin{proof}[Proof of \cref{thm:topcyclicQn}]
$n=1$ is trivial. Assume $n\geq 2$, and as usual denote $u=2^k$ as the largest power of $2$ smaller or equal to $n$. From \cref{lem:binomial} we know that $v_2(c_1(Q_n))=-\min\{v_2(F_n(1)), v_2(F_n(2)),\dots,v_2(F(n))\}$. First we show that we can rule out $F_n(1)$ and all $F_n(a)$ for $a>u$.
\begin{itemize}
    \item \underline{Claim 1:} $v_2(F_n(1))\geq v_2(F_n(2))$.
    From \cref{binomial2}, $v_2(F_n(1))=v_2\left(\binom{n-1}{u-1}\right)-n-k+1$. Similarly $v_2(F_n(2))=v_2\left(\binom{n-2}{u-2}\right)-n-k+1$. Therefore, $v_2(F_n(1))-v_2(F_n(2))=v_2\left(\binom{n-1}{u-1}\right)-v_2\left(\binom{n-2}{u-2}\right)=v_2\left(\frac{n-1}{u-1}\right)\geq0$ since $u$ is even.
    \item \underline{Claim 2:} $v_2(F_n(a)) > v_2(F_n(u))$ for all $a>u$. From \cref{binomial2}, $v_2(F_n(u))=-n+u-1-k$. In the definition of $F_n(a)$, however, all denominators $2(a+i)$ in the sum have $2$-adic value $ \leq k$. Therefore, $v_2(F_n(a))\geq -n+a-k > v_2(F_n(u))$.
\end{itemize}

As a result, we only need to consider the minimum of $v_2(F_n(2)),\dots,v_2(F_n(u))$, which by \cref{binomial2} is
\[v_2(c_1(Q_n))=-\min_{2\leq a\leq u}\left\{v_2(F_n(a))\right\}=n+1+k-\min_{2\leq a\leq u}\left\{a+v_2\left(\binom{n-a}{u-a}\right)\right\}.\]
We then compute the value of the last term in the equation above:
\begin{claim}
$$\min_{2\leq a\leq u}\left\{a+v_2\left(\binom{n-a}{u-a}\right)\right\}=\min_{2\leq a\leq u}\{a+k-v_2(n-a+1),2+k-v_2(n)\}.$$
\end{claim}
To prove the claim, assume the binary expansion of $n$ is $n=2^{p_1}+2^{p_2}+\cdots + 2^{p_d}$ for $0\leq p_1<p_2<\cdots <p_d=k$. Denote $n_i=2^{p_1}+2^{p_2}+\cdots +2^{p_i}$ for $i=1,2,\dots,d-1$ and $n_d=u=2^{p_d}$. We only need to prove the following two subclaims:
\begin{enumerate}
    \item For any $i=1,2,\dots,d-1$, we have
    $$\min_{n_i< a\leq n_{i+1}}\left\{a+v_2\left(\binom{n-a}{u-a}\right)\right\}=n_i+1+k-p_{i+1}=\min_{n_i< a\leq n_{i+1}}\{a+k-v_2(n-a+1)\}.$$
    The first equation comes from the fact that when subtracting $u-a$ from $n-a$, since $n-a$ and $n-u$ are the same in all but the last $p_{i+1}+1$ bits, and $u-a$ have all $1$ except in the last $p_{i+1}+1$ bits, there are always $k-p_{i+1}$ carries in the first $k-p_{i+1}$ bits. By Kummer's Theorem, $v_2\left(\binom{n-a}{u-a}\right)\geq k-p_{i+1}$, and equality is achieved when $a=n_i+1$. The second equation comes from the fact that $v_2(n-a+1)\leq p_{i+1}$ for $a\in(n_i,n_{i+1}]$ and minimum is reached when $a=n_i+1$. 
    \item When $n$ is even, we have
    $$\min_{2\leq a\leq 2^{p_1}}\left\{a+v_2\left(\binom{n-a}{u-a}\right)\right\}=2+k-p_1=\min_{2\leq a\leq 2^{p_1}}\{a+k-v_2(n-a+1),2+k-v_2(n)\}.$$
    The first equation comes from the fact that, similar to case 1, there are always $k-p_1$ carries in the first $k-p_1$ bits, so $v_2\left(\binom{n-a}{u-a}\right)\geq k-p_{1}$, and equality is achieved when $a=2$. The second equation comes from the fact that $v_2(n-a+1)\leq p_{1}$ for $a\in[2,2^{p_1}]$ and the minimum is reached at $2+k-v_2(n)$. 
\end{enumerate}
By combining the two cases, we have the claim, and using the formula from before:
$$v_2(c_1(Q_n))=\max_{2\leq a\leq n}\{v_2(x)+x,v_2(n)+n-1\}$$
\end{proof}

We now work towards the proof of \cref{thm:2ndtopcyclicQn}. Note that finding the $2$nd largest cyclic factor of $K(Q_n)$ is the same as 
finding the maximal additive order in the quotient $\Z$-module $(\Z[x_1,\dots, x_n]/I(Q_n))/\langle x_1 - 1 \rangle$. Here $\langle\cdot\rangle$ represents the $\Z$-submodule generated by ``$\cdot$'' instead of the ring ideal generated by it. The elements of finite additive order are still the elements that are in the kernel of the map $\pi:x_i \mapsto 1$ after the quotient, and it turns out that the second largest Sylow-$2$ cyclic factor is the order of $x_2 - 1$ in the new quotient $\Z$-module, which we prove in the next lemma.

In the rest of this section, equations of polynomials $p(x_1,\dots,x_n)=q(x_1,\dots,x_n)$ means they are equal in the quotient ring $\Z[x_1,\dots, x_n]/I(Q_n)$, and the \emph{order} of a polynomial $p(x_1,\dots,x_n)$ is the additive order of $\overline{p(x_1,\dots,x_n)}\in\Z[x_1,\dots,x_n]/I(Q_n)$, unless explicitly specified otherwise.

To find the order of $x_2-1$ in the quotient $\Z$-module, we want to find $C,D \in \Z$ such that $C(x_1 - 1) + D(x_2 - 1)=0$,  with $D>0$ as small as possible. 
The symmetry of the hypercube allows us to deduce the following nice result even though it is not {\em a priori} obvious which value should $C$ be.

\begin{lemma}
For $2 \le k \le n-1$, the $k$th largest cyclic factor of $K(Q_n)$ is equal to the smallest positive integer $D$ such that $D(x_k - x_1)=0$. In the notation above, this means we can choose $C = -D$.
\end{lemma}
\begin{proof}
To prove this result, we induct on $k$ and adopt the following strategy:
\begin{itemize}
    \item We first show that the order of $x_k-1$ in $(\Z[x_1,\dots, x_n]/I(Q_n))/\langle x_1-1,x_2-1,\dots,x_{k-1}-1 \rangle$ is equal to the smallest positive integer $D$ such that $D(x_k - x_1)=0$.
    \item Then we show that $x_k-1$ has the maximal additive order in the quotient $\Z$-module above. As a result, the $D$ above is exactly the $k$th largest cyclic factor of $K(Q_n)$.
\end{itemize}

The base case is $k = 2$. The order of $x_2-1$ in the quotient $\Z$-module $(\Z[x_1,\dots, x_n]/I(Q_n))/\langle x_1 - 1 \rangle$ is the smallest positive integer $D$ such that $C(x_1 -1) + D(x_2 -1)=0$. Note, however, by symmetry that if $C(x_1 -1) + D(x_2 -1)=0$ then so  $C(x_3 -1) + D(x_1 -1), C(x_3 -1) + D(x_2 -1)=0$. Subtracting shows $D(x_2 - x_1)=0$. Conversely, if $D(x_2 - x_1) = -D(x_1 -1) + D(x_2 -1)=0$ then we can take $C = -D$ and so the order of $x_2-1$ in the quotient $\Z$-module must be the order of $x_2 - x_1$.

Then we show $x_2-1$ is the element with maximal additive order in the quotient $\Z$-module. Namely, we want to show that $D(x_I-1)\equiv 0\mod(x_1-1)$ for any squarefree monomial $x_I$. When $|I|\geq 2$, suppose $x_ix_j|x_I$. Then we have $D(x_I-1)=D(x_i-x_j)\frac{x_I}{x_i}+D\left(\frac{x_I}{x_ix_j}-1\right)=D\left(\frac{x_I}{x_ix_j}-1\right)$ which we can reduce inductively to the base case where $|I|\leq 1$.

For general $k$,
we wish to solve for the minimal $C_k$ such that there exist integer constants $r_1,\dots, r_{k-1}$ such that $r_1(x_1 - 1) + \dots +  r_{k-1}(x_{k-1} -1) + C_k(x_k -1) =0$. Since $k \le n-1$, $x_n$ is not amongst $x_1, \dots, x_k$, and so by symmetry we have that 
$r_1(x_n - 1) + r_2(x_2 -1) + r_3(x_3 - 1) + \dots +  r_{k-1}(x_{k-1} -1) + C_k(x_k -1)=0$ and $r_1(x_n - 1) + r_2(x_2 -1) + r_3(x_3 - 1) + \dots +  r_{k-1}(x_{k-1} -1) + C_k(x_1 -1)=0$. Subtracting yields $C_k(x_k - x_1)=0$, which implies that $C_k$ must just be the order of $x_k - x_1$, as desired. The proof that $C_k$ is the $k$th largest factor follows from the same argument in $k=2$ case. 
\end{proof}

\begin{rem}
The proof above heavily depends on the symmetry of $Q_n$. We wonder if similar arguments work to some extent as long as $G$ has some symmetry with respect to generator classes of edges, so that we might derive similar results for certain general Cayley graphs?
\end{rem}

Note that the order of $x_k - x_1$ is just the order of $x_kx_1 -1$, since $x_1^2 = 1$.
By symmetry, all these elements have the same additive order, so this lemma implies that the $2$nd through $(n-1)$st largest cyclic factors are all the same. It would thus suffice to compute the $2$nd largest cyclic factor.

\begin{proof}[Proof of \cref{thm:2ndtopcyclicQn}]
Using
\cref{orderofmonomial}, we want to find the lowest common denominator of the following numbers
\[
\frac{1}{2^{n-2}}\sum_{\substack{u \cdot v = 1\\ u_1 + u_k = 1}}\frac{1}{\lambda_{u}}.
\]
We will be once again using \cref{lem:spanlemma}, which tells us we need to find the lowest common denominator of \[\frac{1}{2^{n - |S|}}\sum_{\substack{u_S = 1\\ u_1+u_k = 1}}\frac{1}{\lambda_u}.\] By plugging in the values for $\lambda_u$ we obtain the following analogue of \cref{lem:binomial}.

\begin{claim} $c_2(Q_n)$ is the the lowest common denominator of the following rational numbers:
\[\frac{1}{2^{n-a-1}}\sum_{i=0}^{n-a-1}\frac{\binom{n-a-1}{i}}{2(a+i)},\quad \text{for }1\leq a\leq n-1.\]
\end{claim}

The rational numbers in the lemma above are exactly the same as $\frac{1}{2}F_{n-1}(1), F_{n-1}(2), \dots, F_{n-1}(n-1)$ in the $c_1(Q_{n-1})$ case. The fraction in $\frac{1}{2}F_{n-1}(1)$ adds one to the last valuation $v_2(n-1)+(n-1)-1$. Following similar arguments, our factor is just equal to
\[
v_2(c_2(Q_n))=\max_{x < n} \{v_2(x) + x \}.
\]
\end{proof}

\begin{rem}\label{lastconjecture}
One can show that the $n$th largest cyclic factor of $K(Q_n)$ is the additive order of $x_1x_2-1$ in the quotient $\Z$-module $\Z[x_1,\dots, x_n]/(x_1^2-1, \dots, x_n^2-1,\sum x_i-n)/\langle x_1 - 1, x_2-1,\dots, x_n-1\rangle$. This corresponds to the smallest positive integer $D$ such that there exist integer constants $r_1,\dots,r_n$ where $D(x_1x_2-1)+r_1(x_1-1)+\cdots+r_n(x_n-1)=0$, but we are unable to nail down the exact values for $r_1,\dots,r_n$. However, we believe the value of $D$ should be given by the following conjecture, which is supported by data for $n\leq 11$.
\end{rem}

\begin{conjecture}
\label{conj:nthtopcyclicQn}
For $n\geq 3$, we have
$$v_2(c_n(Q_n))=\max\left\{\max_{x<n-1}\{v_2(x)+x\}, v_2(n-1)+n-3\right\}.$$
\end{conjecture}

\section{Conclusion and Remaining Conjectures}
In this paper we have analyzed the sandpile groups of Cayley graphs by careful looking at the underlying structure of cyclic factors in the corresponding polynomial ring. While our results partially characterize $\on{Syl}_2(K(G))$, determining the complete structure still seems out of reach at this moment.

One possible route to extract finer information could be via finding Gr\"{o}bner bases for these polynomials. However, we were not able to find strong patterns in the Gr\"{o}bner bases even for $K(Q_n)$, and had difficulty working over $\Z$ with these objects. Nevertheless, this appears to be a place for further exploration.

We list remaining conjectures we have gathered based on data. %

\begin{conjecture}
The number of Sylow-$2$ cyclic factors $d(M)$ is odd unless all of the eigenvalues have the same power of $2$, in which case $d(M) = 2^n - 2$.
\end{conjecture}

\begin{conjecture}
When the greatest common divisor of all generator multiplicities in $M$ is $1$, the sandpile group depends only on the set of eigenvalues generated by $M$, and independent of the specific choice of $M$.
\end{conjecture}

\begin{conjecture}
Two Cayley graphs have the same sandpile group if and only if their generator multiplicities are the same up to $GL$-equivalence.
\end{conjecture}

\begin{conjecture}
The Sylow-$2$ component of the sandpile group for $Q_{2^k - 1}$ and $Q_{2^k}$ differs as follows: $Syl_2(K(Q_{2^k}))$ equals a top cyclic factor as determined in \cref{topcyclicQnsection} and then the remaining factors come from taking $Syl_2(K(Q_{2^k - 1}))$ and doubling the multiplicity of each factor. That is, we have
\[
\Syl_2(K(Q_{2^k})) \cong \Syl_2(K(Q_{2^k-1}))^2 \times \Z/(2^{2^k + k -1}\Z).
\]
\end{conjecture}

We note that \cite[Remark~5.2]{chandler2017smith} allows one to transfer some information from the (well-understood) Smith group of $Q_n$ to $\Syl_2(K(Q_n))$.
Since the amount of information depends on the value of $v_2(n)$,
their method provides some non-trivial information for the L.H.S. of the above conjecture, but unfortunately provides essentially nothing for the R.H.S.. Hence, proving the conjecture (assuming not by computing the sandpile groups directly) can help building a bridge there.

\begin{table}[ht]
\centering
\bgroup
\def\arraystretch{1.5}
\begin{tabular}{|c|l|}
\hline
$n$ & $\on{Syl}_2(K(Q_n))$\\ \hline
$2$ & $\Z_4$\\
$3$ & $\Z_2\Z_8^2$\\
$4$ & $\Z_2^2\Z_8^4\Z_{32}$\\
$5$ & $\Z_2^6\Z_8^4\Z_{16}\Z_{64}^4$\\
$6$ & $\Z_2^{12}\Z_4^4\Z_8\Z_{32}^4\Z_{64}^{10}$\\
$7$ & $\Z_2^{28}\Z_4\Z_{16}^8\Z_{32}^6\Z_{64}^{14}\Z_{128}^6$\\
$8$ & $\Z_2^{56}\Z_4^2\Z_{16}^{16}\Z_{32}^{12}\Z_{64}^{28}\Z_{128}^{12}\Z_{1024}$\\
$9$ & $\Z_2^{120}\Z_4^{10}\Z_{16}^{16}\Z_{32}^{26}\Z_{64}^{48}\Z_{128}^{26}\Z_{512}\Z_{2048}^8$\\
$10$ & $\Z_2^{240}\Z_4^{36}\Z_8^{26}\Z_{32}^{16}\Z_{64}^{148}\Z_{256}\Z_{1024}^{26}\Z_{2048}^{18}$\\
$11$ & $\Z_2^{496}\Z_4^{66}\Z_8^{32}\Z_{16}^{100}\Z_{64}^{164}\Z_{128}\Z_{512}^{100}\Z_{2048}^{64}$\\
\hline
\end{tabular}
\egroup
\caption{Sylow-$2$ cyclic factors of $K(Q_n)$ for $n\leq 11$ (table from \cite{bai})}
\label{tab:signedsymmetric}
\end{table}

\section*{Acknowledgements}

A portion of this research was carried out as part of the 2018 REU program at the School of Mathematics at University of Minnesota, Twin Cities. The first three authors are grateful for the support of NSF RTG grant DMS-1745638 and making the program possible, as well as to thank Professor Victor Reiner for providing both guidance and independence in their research efforts. The first three authors would also like to thank Eric Stucky for his edits to this paper and Amal Mattoo for his contributions to this research.
The fourth author was supported by the Trond Mohn Foundation project ``Algebraic and Topological Cycles in Complex and Tropical Geometries''.
He thanks Victor Reiner for introducing him to the first three authors, and Kevin Iga for the discussion on critical groups.
All authors thank the anonymous referee for the thorough reading and comments.

\printbibliography

\end{document}